\documentclass{article}
\usepackage{latexsym, amssymb, enumerate, amsmath,amsthm,pdfsync, multicol}
\usepackage{fullpage}
\usepackage[utf8]{inputenc} 
\usepackage[T1]{fontenc}
\usepackage{lmodern}
\usepackage{graphicx}
\usepackage[english]{babel}
\usepackage[numbers]{natbib}
\usepackage{url}
\usepackage{esint}
\usepackage{dsfont}
\usepackage{cleveref}
\usepackage[normalem]{ulem}
\usepackage{mathrsfs}
\usepackage{pifont}

\crefname{hyp}{hypothesis}{hypotheses}
\crefname{thm}{theorem}{theorems}
\crefname{lem}{lemma}{lemmas}
\crefname{cor}{corollary}{corollaries}
\crefname{prop}{proposition}{propositions}
\Crefname{theorem}{Theorem}{Theorems}

\newcommand{\R}{\mathbb{R}}

\newcommand{\1}{\mathds{1}}

\renewcommand{\epsilon}{\varepsilon}

\newcommand{\undu}{\underline{u}}

\newtheorem{thm}{Theorem}[section]
\newtheorem{prop}[thm]{Proposition}

\newtheorem{hyp}[thm]{Hypothesis}

\newtheorem{remark}[thm]{Remark}
\newtheorem{theorem}[thm]{Theorem}

\def\j{{\mathcal{J}}}

\def\D{{\mathcal{D}}}
\newcommand{\opd}[1]{\D[\,{#1}\,]}

\author{Emeric Bouin \footnote{CEREMADE - Universit\'e Paris-Dauphine, UMR CNRS 7534, Place du Mar\'echal de Lattre de Tassigny, 75775 Paris Cedex 16, France. E-mail: \texttt{bouin@ceremade.dauphine.fr}}
\and Jérôme Coville \footnote{UR 546 Biostatistique et Processus Spatiaux, INRAE, Domaine St Paul Site Agroparc, F-84000 Avignon, France. E-mail: \texttt{jerome.coville@inrae.fr}}
\and Guillaume Legendre \footnote{CEREMADE - Universit\'e Paris-Dauphine, UMR CNRS 7534, Place du Mar\'echal de Lattre de Tassigny, 75775 Paris Cedex 16, France. E-mail: \texttt{legendre@ceremade.dauphine.fr}}}

\begin{document}
\title{A simple flattening lower bound for solutions to some linear integrodifferential equations}
\maketitle

\begin{abstract}
Estimates on the asymptotic behaviour of solution to linear integro-differential equations are fundamental in understanding the dynamics occuring in many nonlocal evolution problems. They are usually derived by using precise decay estimates on the heat kernel of the considered diffusion process. In this note, we show that for some generic jump diffusion and particular initial data, one can derive a lower bound of the asymptotic behaviour of the solution using a simple PDE argument. This is viewed as an independant preliminary brick to study invasion phenomena in nonlinear reaction diffusion problems.
\end{abstract}


\section{Introduction}\label{sec:Intro}
In this note, we are interested in asymptotic behaviour of the solution to a generic nonlocal integro-differential equation of the form:
\begin{equation}\label{eq:main}
\frac{\partial u}{\partial t}(t,x) =\opd{u}(t,x),\quad t\in(0,+\infty),\ x\in\R,
\end{equation}
where $\opd{\cdot}$ is a diffusion operator with a kernel of convolution type $J$, that is
\[
\opd{u}(t,x):=\text{P.V.}\left(\int_{\R}[u(t,x-z)-u(t,x)]J(z)\,dy\right).
\]
This equation is complemented by an initial condition
\begin{equation}\label{eq:initial}
u(0,x)=u_0(x),\quad x\in\R,
\end{equation}
to form an evolution problem. We assume that the initial data $u_0$ belongs to $\mathscr{C}(\R,\R_+)\cap L^{\infty}(\R)$ and satisfies the following hypothesis.

\begin{hyp}\label{hyp:u0}
There exist $a>0$ and $b\in \R$ such that $ u_0 \ge a \mathds{1}_{(-\infty,b]}$.
\end{hyp}
Morever, we assume the kernel $J$ is a nonnegative function, satisfying the following properties.

\begin{hyp}\label{hyp:J}
Let $s$ be a positive real number. The kernel $J$ is symmetric and such that there exist positive constants $\j_0,\j_1$ and $R_0$, the latter being larger than $1$, such that 
\[
\int_{|z|\le 1}J(z)|z|^2\,dz\le 2\j_1 \quad \text{ and }\quad \frac{\j_0}{|z|^{1+2s}}\ge J(z)\mathds{1}_{|z|>1}(z)\ge \frac{\j_0^{-1}}{|z|^{1+2s}}\1_{\{|z|\ge R_0\}}.
\]
\end{hyp}

The operator $\opd{\cdot}$ can be seen as the infinitesimal generator of a generic symmetric jump process \cite{Bogdan2014}, and appears, for instance, in population dynamics where it describes the dispersion of individuals of a population modelled by the density $u$. Roughly speaking, the value $J(z)$ represents the probability of a jump of size $z$, which makes the tails of the convolution kernel of crucial importance in quantifying the dynamics. One may readily notice that the assumptions on $J$ allow to cover the two broad types of integro-differential operators usually considered in the literature, which are the fractional Laplace operator $(-\Delta)^s u$ and a standard convolution operator with integrable kernel, often written $J \star u - u$, respectively. 

The characterisation of the asymptotic behaviour of solutions to linear diffusion problems such as \eqref{eq:main}-\eqref{eq:initial} is a classical question, which can be answered for particular classes of Levy processes through the time and space scaling properties of their associated heat kernel, see for instance \cite{Polya1923,Blumenthal1960} for the fractional Laplace operator and \cite{Barlow2009,Bogdan2014,Cygan2017,Grzywny2019,Kaleta2019,Knopova2013,Kolokoltsov2000} in the case of more general Levy processes. Indeed, when such a heat kernel exists, the solution $u$ to \eqref{eq:main}-\eqref{eq:initial} is given by
\[
\forall t\in\in[0,+\infty),\ \forall x\in\mathbb{R},\qquad u(t,x)=\int_{\R}p(t,y)u_0(x-y)\,dy,
\]
where $p$ is the solution to
\begin{align*}
&\frac{\partial p}{\partial t}(t,x) =\opd{p}(t,x),\quad t\in(0,+\infty),\ x\in\R,\\
&p(0,x)=\delta_0(x),\quad x\in\R,
\end{align*}
where $\delta_0$ is the Dirac delta distribution. It then follows that the asymptotics of $u$ can be derived from the time and space scaling properties of $p$. For example, for the fractional Laplace operator $(-\Delta)^s$, it is well known (\cite{Blumenthal1960,Chen2003}) that the heat kernel $p_s$ satisfies, for some positive constant $C_1$, the following scalings
\[
\forall t\in(0,+\infty),\ \forall x\in\mathbb{R},\ \frac{{C_1}^{-1}}{t^{\frac{1}{2s}}[1+|t^{-\frac{1}{2s}} x|^{1+2s} ] }\leq p_s(t,x)\leq \frac{C_1}{t^{\frac{1}{2s}}[1+|t^{-\frac{1}{2s}} x|^{1+2s}]}.
\]
As a consequence, a solution to \eqref{eq:main}-\eqref{eq:initial} with a fractional Laplace operator satisfies
\begin{align*}
u(t,x)\geq\int_{\mathbb{R}} u_0(x-y)\frac{{C_1}^{-1}}{t^{\frac{1}{2s}}[1+|t^{-\frac{1}{2s}} y|^{1+2s} ] }\,dy &\geq\int_{x-b }^{+\infty} \frac{a\,{C_1}^{-1} }{t^{\frac{1}{2s}}[1+|t^{-\frac{1}{2s}} y|^{1+2s} ] }\,dy\\
&=\int_{t^{-\frac{1}{2s}}(x-b)}^{+\infty} \frac{a\,{C_1}^{-1} }{1+|z|^{1+2s} }\,dz\\
&=\int_{t^{-\frac{1}{2s}}(x-b)}^{+\infty}\left(\frac{a\,C_1^{-1}}{|z|^{1+2s}}-\frac{a\,{C_1}^{-1}}{(1+|z|^{1+2s})|z|^{1+2s} }\right)\,dz,
\end{align*}
which, in particular, provides the following estimate:
\[
\forall t\in(0,+\infty),\ \lim_{x\to+\infty} \frac{x^{2s}}{t}u(t,x)\ge \lim_{x\to+\infty} \frac{x^{2s}}{t}\left( \frac{t a\,{C_1}^{-1}}{(x-b)^{2s}} - \int_{t^{-\frac{1}{2s}}(x-b)}^{+\infty}\frac{a\,{C_1}^{-1} }{(1+|z|^{1+2s})|z|^{1+2s} }\right)=a\,{C_1}^{-1}.
\]

It is expected that an analogous estimate holds for solutions of \eqref{eq:main}-\eqref{eq:initial} when $\opd{\cdot}$ is a generic operator whose kernel satisfies \Cref{hyp:J}, and we shall here obtain such a flattening estimate directly from the problem, without any further restriction on the considered L\'evy measure other than those given in \Cref{hyp:J} or any knowledge of the associated heat kernel. Let us state precisely this result.

\begin{theorem}\label{th:main}
Assume that the kernel $J$ and the initial datum $u_0$ satisfy \Cref{hyp:J} and \ref{hyp:u0}, respectively. Then, there exists a constant $\kappa$, depending on $J$ and $u_0$, such that the solution $u$ to \eqref{eq:main}-\eqref{eq:initial} has the following asymptotic flattening behaviour at infinity:
\[
\forall t\in(0,+\infty),\ \lim_{x\to +\infty} x^{2s}u(t,x)\ge \kappa t.
\]
\end{theorem}

One of the main applications of the above type of estimate arises naturally in the study of propagation phenomena described by some semi-linear equation of the form
\begin{equation}\label{eq:mono}
\frac{\partial u}{\partial t}(t,x) =\opd{u}(t,x) + f(u(t,x)),\quad t\in(0,+\infty),\ x\in\R,
\end{equation}
where $f$ is a nonlinearity describing the local dynamics of the modelled system. On way to capture both the evolution and speed of transition in the resulting problem is by means of the construction of super- and sub-solution that mimic the essential features (reaction and dispersal) of the system. The natural time and space scalings of the equation play an important role in these constructions. When the nonlinearity $f$ is nonnegative, the solution to \eqref{eq:main}-\eqref{eq:initial} is a trivial sub-solution to the problem, giving rise to a first lower bound on the decay of the solution. In order to achieve a more detailed description of the dynamics, a more sophisticated sub-solution needs to be constructed and the estimate obtained in \Cref{th:main} simplifies such a construction by first deriving the natural time and space scalings of the solution of the semi-linear equation and next allowing to compare this solution with a sub-solution via the use of a parabolic comparison principle.

For instance, when the non-linearity is monostable with, for instance, $0$ and $1$ as equilibria, i.e., $f(0)=f(1)=0$, one can try to understand the dynamics of the solution to \eqref{eq:mono}-\eqref{eq:initial} by considering a level set of height $\lambda$ in $(0,1)$ and a sub-solution $\undu$ such that $\undu\le \lambda$ and satisfying the following decay at infinity:
\begin{equation}\label{eq:mono-subsol}
\forall\theta\in(0,+\infty),\ \forall t\in(0,+\infty),\ \lim_{x\to+\infty}\frac{x^{2s}}{\theta t}\undu(t,x)<+\infty. 
\end{equation}
If such a sub-solution exists, then, for a positive real number $\theta_0$ and a positive time $t_0$, there exists a positive constant $C_0$ such that 
\[
\lim_{x\to+\infty}\frac{x^{2s}}{\theta_0 t_0}\undu(t_0,x)\le C_0.
\]
From \Cref{th:main}, one also has
\[
\forall t\in(0,+\infty),\ \lim_{x\to+\infty} \frac{x^{2s}}{t}u(t,x)\ge \kappa\,a,
\]
so that there exists a positive real number $t'$ such that $\kappa a t'>2\theta_0 t_0C_0$. It follows that $\undu(t_0,x)\le u(t',x)$ for $x$ large enough, say $x>x_0>0$ and, using \Cref{th:main}, one has $\undu(t_0,x)\le u(t,x)$ for $t\ge t'$ and $x>x_0>0$. Combining this with the invasion property usually satisfied by the solution $u$ in such context, i.e. a property asserting that $u(t,x)$ tends to $1$ as $t$ tends to infinity and uniformly in $x$ in $(-\infty,A]$ for any real number $A$, we may find a time $t''>t'$ such that $u(t'',x)\ge \undu (t_0,x)$ for $x\in\R$. Due to a parabolic comparison principle, it follows that $u(t''+t-t_0,x)\ge \undu(t,x)$ for $t\ge t_0$ and $x\in\R$, implying that the level set of height $\lambda$ of the solution travels at a speed at least equal to that of the sub-solution. 

\medskip

The present note is organised as follows. First, some comparison principles are recalled and a useful \textit{a priori} bound is derived. The argument needed to prove \Cref{th:main} is next developed. 

\section{Preliminaries}
Let us start by recalling the different comparison principles that we will use throughout this note. 

\begin{theorem}[standard comparison principle]\label{SCP}
Assume that the kernel $J$ satisfies \Cref{hyp:J} and let $u$ and $v$ be two functions in $\mathscr{C}^1(\R_+^*,\mathscr{C}(\R))\cap\mathscr{C}(\R_+,\mathscr{C}(\R))$, satisfying, for some positive real number $T$,
\begin{align*}
&\frac{\partial u}{\partial t}(t,x)\ge \opd{u}(t,x),&\quad t\in (0,T),\ x\in \R,\\
&\frac{\partial v}{\partial t}(t,x)\le \opd{v}(t,x),&\quad t\in (0,T),\ x\in \R,\\
&u(0,x)\ge v(0,x),&\quad x\in \R.
\end{align*} 
Then, one has $u(t,x)\ge v(t,x)$ for $t\in (0,T)$, $x\in \R$. 
\end{theorem}

The second comparison principle works with a subset in space.
\begin{theorem}[adapted comparison principle]\label{ACP}
Assume that the kernel $J$ satisfies \Cref{hyp:J} and let $u$ and $v$ be two functions in $\mathscr{C}^1(\R_+^*,\mathscr{C}(\R))\cap \mathscr{C}(\R_+^*,\mathscr{C}(\R))$, satisfying, for some $0\le t_0<t_1$ and $R_0\in \R$,
\begin{align*}
&\frac{\partial u}{\partial t}(t,x)\ge \opd{u}(t,x),&\quad t\in (t_0,t_1),\ x\in [R,+\infty),\\
&\frac{\partial v}{\partial t}(t,x)\le \opd{v}(t,x),&\quad t\in (t_0,t_1),\ x\in [R,+\infty),\\
&u(t,x)> v(t,x),&\quad t\in (t_0,t_1),\ x\in (-\infty,R],\\
&u(t_0,x)\ge v(t_0,x),&\quad x\in \R,.
\end{align*} 
Then, one has $u(t,x)\ge v(t,x)$ for $t\in(t_0,t_1)$, $x\in \R$.
\end{theorem}
The proofs of \Cref{SCP} and \Cref{ACP} are rather standard and will be omitted here, but the reader can refer to \cite{Brasseur2021,Zhang2021} for some ideas.

Some \textit{a priori} estimates on the solution to \eqref{eq:main}-\eqref{eq:initial} are also needed.

\begin{prop}\label{prop:obs}
Assume that the kernel $J$ and the initial datum $u_0$ satisfy \Cref{hyp:J} and \ref{hyp:u0}, respectively. Let $u$ be a positive solution to \eqref{eq:main}-\eqref{eq:initial}. Then, one has
\[
\forall t\in[0,+\infty),\ \forall x\in(-\infty,b),\ u(t,x)> \frac{a}{2}.
\] 
\end{prop}

\begin{proof}
Let $\rho$ be a smooth symmetric mollifier of unit mass, i.e., a nonnegative even function $\rho$ in $\mathscr{C}_c^{\infty}(\R)$ such that $\int_{\R}\rho(z)\,dz=1$, and consider the solution $v$ to the problem: 
\begin{align}
&\frac{\partial v}{\partial t}(t,x)=\opd{v}(t,x),\quad t\in(0,+\infty),\ x\in\R,\label{eq:lin}\\
&v(0,x)=v_0(x):=a\rho\star\mathds{1}_{(-\infty,b]}(x),\ x\in\R.\label{eq:lin-CI}
\end{align} 
Since by construction $u_0\ge a \mathds{1}_{(-\infty,b]}$ and is continuous, we may assume that, up to a rescaling of $\rho$, $u_0\ge v_0$. Therefore, by the comparison principle in \Cref{SCP}, we have $u(t,x)\ge v(t,x)$ for $t\ge 0$ and $x \in \R$. Moreover since $v_0$ is monotone non increasing and equation \eqref{eq:lin} is invariant under translation in space, we deduce that, for $t>0$, the function $v(t,\cdot)$ is monotone decreasing.\\
Let us now observe that the function $\bar v(t,x):=v(t,x+b)+v(t,b-x)$ satisfies equation \eqref{eq:lin}, with the initial condition $\bar v(0,\cdot)=v_0(x+b)+v_0(b-x)$. By a straightforward change of variables, we have
\begin{align*}
\forall x\in\R,\ v_0(x+b)+v_0(b-x)&=a\left(\int_{-\infty}^{0}\rho(x-y)\,dy + \int_{-\infty}^0\rho(-x-y)\,dy\right)\\
&= a\left(\int_{-\infty}^{0}\rho(x-y)\,dy + \int_{-\infty}^0\rho(x+y)\,dy\right)\\
&=a.
\end{align*}
By the uniqueness of the mild solution of the initial value problem (see for instance Theorem 4.3 in \cite{Pazy1983}), this implies that $\bar v\equiv a$. As a consequence, we deduce that $2\,v(t,b)=a$ for $t>0$ and, since $v(t,\cdot)$ is decreasing, we have $v(t,x)>\frac{a}{2}$ for $t>0$ and $x<b$, thus ending the proof.
\end{proof}

\begin{remark}
When the semigroup generated by $\opd{\cdot}$ is regularising, i.e. the solution $u$ belongs to $\mathscr{C}^1(\R_+^*,\mathscr{C}(\R_+))$ whenever $u_0$ is in $L^{\infty}(\R)$, the above argument holds for initial date that are step functions, e.g. $v(0,\cdot)=a\mathds{1}_{(-\infty,b]}$. This is not necessarily the case for non-regularising semigroups such as those related to a convolution operator with a continuous integrable kernel. In such situations, the regularity of $v$ is the same as of $u_0$ and $v$ has a jump discontinuity at the point $b$, preventing an evaluation $\bar v$ at this point.
\end{remark} 

\begin{remark}
Note that the proof of the above estimate relies solely on an elementary use of the comparison principle for the evolution problem. It is valid, in full generality, as soon as the considered semigroup possesses some basic properties, such as mapping a continuous function to another one, satisfying a comparison principle and having a translation invariant infinitesimal generator. In particular, it holds true for semigroups generated by operators satisfying the Bony--Courr\`ege--Priouret maximum principle \cite{Bony1968}, characterised by an elliptic part and a L\'evy-type part, the latter being associated with a symmetric L\'evy measure, that is a nonnegative, nonzero measure $\nu$ on $\R$, satisfying $\nu(\{0\})=0$, $\nu(-A)=\nu(A)$ for every Borel set in $\R$, and $\int_{\R}\min\{1,z^2\}\nu(dz)<+\infty$. 
\end{remark}

\section{Proof of \Cref{th:main}}\label{sec:proof}
 Having this preliminary estimate at hand, we can now prove \Cref{th:main}. Our strategy is to construct an adequate subsolution. Let $w$ be the parametric function defined by
\[
w(t,x)=
\begin{cases}
\frac{1}{2} &  t\in(0,+\infty),\ x\in(-\infty,0],\\
\frac{\kappa t}{x^{2s}+2 \kappa t} & t\in(0,+\infty),\ x\in(0,+\infty),
\end{cases}
\]
with $\kappa=\frac{{\j_0}^{-1}}{8s}$.

Let us estimate $\opd{w}$. Let $R>1$ to be chosen later. Since $J$ satisfies \Cref{hyp:J}, we have by a direct computation, for $t\in(0,+\infty)$ and $x\in[R_0+R,+\infty)$,
\begin{align*}
\opd{w}(t,x)&=\int_{-\infty}^{-R} [w(t,x+z)-w(t,x)]J(z)\,dz +\int_{-R}^{R}[w(t,x+z)-w(t,x)]J(z)\,dz \\
&\qquad \qquad +\int_{R}^{+\infty}[w(t,x+z)-w(t,x)]J(z)\,dz\\
&\ge \int_{-\infty}^{-R} [w(t,x+z)-w(t,x)]J(z)\,dz +\int_{-R}^{R}[w(t,x+z)-w(t,x)]J(z)\,dz -w(t,x)\int_{R}^{+\infty}J(z)\,dz\\
&=\int_{-R}^{R} [w(t,x+z)-w(t,x)]J(z)\,dz +\int_{-\infty}^{-x} [w(t,x+z)-w(t,x)]J(z)\,dz\\
& \qquad \qquad+\int_{-x}^{-R}[w(t,x+z)-w(t,x)]J(z)\,dz -w(t,x)\int_{R}^{+\infty}J(z)\,dz\\
&\ge \int_{-R}^{R} [w(t,x+z)-w(t,x)]J(z)\,dz +\left[\frac{1}{2}-w(t,x)\right] \int_{x}^{+\infty}J(z)\,dz -w(t,x)\int_{R}^{+\infty}J(z)\,dz\\
&\ge \int_{-R}^{R} [w(t,x+z)-w(t,x)]J(z)\,dz +\frac{{\j_0}^{-1}}{2s}\left[\frac{1}{2}-w(t,x)\right] \frac{1}{ x^{2s}} - \frac{\j_0}{2sR^{2s}} w(t,x).
\end{align*}
 
The remaining integral can be estimated using the regularity and the convexity with respect to space of $w$, together with the symmetry of $J$. Indeed, since $w(t,\cdot)$ belongs to $\mathscr{C}^{1}(\R_+)$ for $t>0$, we have, for $x\ge R_0+R$,
\[
\forall t\in(0,+\infty),\ \forall x\in[R_0+R,+\infty),\ w(t,x+z)-w(t,x)=z\int_{0}^1\partial_x w(t,x+\tau z)\,d\tau,
\]
and thus
\[
\int_{-R}^{R} [w(t,x+z)-w(t,x)]J(z)\,dz=\int_{-R}^{R}\int_{0}^1 \partial_xw(t,x+\tau z)J(z)z\,dz.
\]
The kernel $J$ being symmetric, we have
\[
\int_{-R}^{R} [w(t,x+z)-w(t,x)]J(z)\,dz=\int_{-R}^{R}\int_{0}^1 [\partial_xw(t,x+\tau z)-\partial_xw(t,x)]J(z)z\,d\tau dz,
\]
which can be rewritten as 
\[
\int_{-R}^{R} [w(t,x+z)-w(t,x)]J(z)\,dz=\int_{-R}^{R}\int_{0}^1\int_{0}^1 \partial_{xx}w(t,x+\tau \sigma z)J(z)\tau z^2\,d\sigma d\tau dz,
\]
using that $w(t,\cdot)$ belongs to $\mathscr{C}^2(\R_+)$ for $t>0$. Since $w$ is convex with respect to space in $\R_+$, the latter integral is positive and we get 
\[
\opd{w}(t,x)\ge \frac{{\j_0}^{-1}}{2s}\left[\frac{1}{2}-w(t,x)\right] \frac{1}{ x^{2s}}-\frac{\j_0}{2sR^{2s}} w(t,x),\quad t\in(0,+\infty),\ x\in[R_0+R,+\infty).
\]
Altogether, we then have for $t>0$ and $x\ge R_0+R$,
\[
\frac{\partial w}{\partial t}(t,x) -\opd{w}(t,x)\le \frac{\kappa x^{2s}}{(x^{2s}+2\kappa t)^2} -\frac{{\j_0}^{-1}}{2s}\left[\frac{1}{2}-w(t,x)\right] \frac{1}{ x^{2s}} +\frac{\j_0}{2s R^{2s}} w(t,x),
\]
which, by using the definition of $w$, yields 
\begin{align*}
\frac{\partial w}{\partial t}(t,x) -\opd{w}(t,x)&\le \frac{\kappa x^{2s}}{(x^{2s}+2\kappa t)^2} -\frac{{\j_0}^{-1}}{2s}\left[\frac{1}{2}-\frac{\kappa t}{x^{2s} +2\kappa t}\right] \frac{1}{ x^{2s}} +\frac{\j_0}{2s R^{2s}} \frac{\kappa t}{x^{2s}+2\kappa t}\\
&\le \frac{\kappa x^{2s}}{(x^{2s}+2\kappa t)^2} -\frac{{\j_0}^{-1}}{4s}\frac{1}{x^{2s} +2\kappa t} +\frac{\j_0}{2s R^{2s}} \frac{\kappa t}{x^{2s}+2\kappa t}\\
&\le \frac{1}{(x^{2s}+2\kappa t)}\left(\kappa  -\frac{{\j_0}^{-1}}{4s} +\frac{\j_0 \kappa t}{2s R^{2s}}\right).
\end{align*}
Since $\kappa=\frac{{\j_0}^{-1}}{8s}$, we end up with
\begin{align*}
\frac{\partial w}{\partial t}(t,x) -\opd{w}(t,x)&\le \frac{1}{x^{2s}+2\kappa t} \left( -\frac{{\j_0}^{-1}}{8s} + \frac{\j_0\kappa t}{2s R^{2s}}\right).
\end{align*}
For any $C>0$, let us define $t^*:= \frac{2C}{\kappa}$ and choose $R$ large enough says $R\ge R_C:= \left(8C{\j_0}^2\right)^{\frac{1}{2s}}$. From the above computations, it then follows that
\begin{equation}\label{bcl2:eq-subsol-flatnonlin}
\frac{\partial w}{\partial t}(t,x) -\opd{w}(t,x)\le 0, \quad t \in (0,t^*),\ x\in[R_0+R_C,+\infty).
\end{equation}

Equipped with this subsolution, let us now conclude. Using \Cref{prop:obs}, there exists $a\in\R_+^*$ and $b\in \R$ such that $u(t,x)>\frac{a}{2}$ for $t>0$ and $x<b$. By definition of $w$, we thus have $u(t,x-R_0-R_C-b)>\frac{a}{2}\ge aw(t,x)$ for $t>0$ and $x\le R_0+R_C$. Therefore, setting $\tilde u(t,x):=u(t,x-R_0-R_C-b)$, we have 
\begin{align*}
&\frac{\partial\tilde u}{\partial t}(t,x) -\opd{\tilde u}(t,x)= 0,\quad t \in (0,t^*),\ x\in \R,\\
& \tilde{u}(t,x) \ge aw(t,x),\quad t \in [0,t^*],\ x\le R_0+R_C,\\
& \tilde{u}(0,x) > aw(0,x),\quad x\in \R.
\end{align*}

Using \Cref{ACP}, it follows that for $t\in(0,t^*)$ and $x\in\R$, one has $\tilde u(t,x)\ge aw(t,x)$ and therefore, by the definition of $t^*$,
\[
u\left(\frac{t^*}{2},x\right)\ge \frac{aC}{(x+R_0+R_C+b)^{2s}+2 C}.
\]
It follows that 
\[
\lim_{x\to +\infty} x^{2s}u\left(\frac{t^*}{2},x\right)\ge \lim_{x\to +\infty} \frac{Cax^{2s}}{(x+R_0+R_C+b)^{2s}+2 C}=Ca,
\]
or, equivalently, that, for all positive real number $C$,
\[
\lim_{x\to +\infty} x^{2s}u\left(\frac{C}{\kappa},x\right)\ge Ca.
\]
This implies that
\[
\forall t\in(0,+\infty),\ \lim_{x\to +\infty} x^{2s}u(t,x)\ge a\kappa t,
\]
thus ending the proof.


\bibliographystyle{plain}
\bibliography{only_flat.bib}
\end{document}